\documentclass[12pt,a4paper]{amsart}
\usepackage{a4wide}
\usepackage[utf8]{inputenc}
\usepackage[T1]{fontenc}

\usepackage{amsmath,amssymb}
\usepackage{bbm}

\usepackage{xcolor}
\usepackage{booktabs}
\usepackage{hyperref}
\hypersetup{
    colorlinks,
    citecolor=black,
    filecolor=black,
    linkcolor=blue,
    urlcolor=black
}
\usepackage{stmaryrd}
\usepackage{url}
\usepackage{longtable}
\usepackage[figuresright]{rotating}
\usepackage{amsthm}
\usepackage{enumerate}
\usepackage{geometry}
\newtheorem{definition}{Definition}
\newtheorem{theorem}[definition]{Theorem}
\newtheorem{proposition}[definition]{Proposition}
\newtheorem{corollary}[definition]{Corollary}

\newtheorem{lemma}[definition]{Lemma}
\newtheorem{fact}[definition]{Fact}
\newtheorem{remark}[definition]{Remark}
\newtheorem{example}[definition]{Example}
\newtheorem{question}[definition]{Question}
\newtheorem*{claim*}{Claim}

\newcommand{\0}{\emptyset}
\newcommand{\mc}{\mathcal}

\newcommand{\RR}{\mathbb{R}}

\newcommand{\NN}{\mathbb{N}}
\newcommand{\QQ}{\mathbb{Q}}

\newcommand{\symdif}{\triangle}

\newcommand{\IFF}{\Leftrightarrow}
\newcommand{\IMP}{\Rightarrow}

\newcommand{\Ii}{\mc{I}}
\newcommand{\Jj}{\mc{J}}

\newcommand{\Mm}{\mc{M}}
\newcommand{\Nn}{\mc{N}}
\newcommand{\Ff}{\mc{F}}
\newcommand{\Dd}{\mc{D}}

\newcommand{\Aa}{\mc{A}}
\newcommand{\Bb}{\mc{B}}
\newcommand{\Cc}{\mc{C}}

\newcommand{\Gg}{\mc{G}}

\newcommand{\bez}{\backslash}
\newcommand{\se}{\subseteq}
\newcommand{\sen}{\subsetneq}
\newcommand{\es}{\supseteq}
\newcommand{\nes}{\supsetneq}

\newcommand{\rest}{\restriction}

\newcommand{\wt}{\widetilde}

\newcommand{\quotient}[2]{\raisebox{.2em}{$#1$}\big/\raisebox{-0.2em}{$#2$}}

\newcommand{\tn}[1]{\textnormal{#1}}
\newcommand{\ti}[1]{\textit{#1}}
\newcommand{\cf}{\textnormal{cf}}

\newcommand{\Bor}{\tn{Bor}}
\newcommand{\ccc}{\tn{ccc}}

\newcommand{\cov}{\tn{cov}}

\newcommand{\splitt}{\textnormal{split}}
\newcommand{\omsplit}{\textnormal{$\omega$-split}}
\newcommand{\Succ}{\textnormal{Succ}}
\newcommand{\succe}{\textnormal{succ}}
\newcommand{\stem}{\textnormal{stem}}

\def\ca{\mathcal{A}}
\def\cb{\mathcal{B}}
\def\Bb{\mathcal{B}}

\def\cf{\mathcal{F}}

\def\bbq{\mathbb{Q}}
\def\bbr{\mathbb{R}}

\def\c{\mathfrak{c}}

\def\w{\omega}

\def\cov{\rm cov}

\title{Ideals with Smital properties}

\author{Marcin Michalski}
\email{marcin.k.michalski@pwr.edu.pl}
	
\author{Robert Rałowski}
\email{robert.ralowski@pwr.edu.pl}

\author{Szymon Żeberski}
\email{szymon.zeberski@pwr.edu.pl}

\address{Marcin Michalski, Robert Rałowski, Szymon Żeberski, Wrocław University of Science and Technology, Faculty of Pure and Applied Mathematics, Wybrzeże Wyspiańskiego 27,	50-370 Wrocław, Poland}

\date{}

\makeatletter
\@namedef{subjclassname@2020}{%
  \textup{2020} Mathematics Subject Classification}
\makeatother

\begin{document}



\begin{abstract}
	A $\sigma$-ideal $\Ii$ on a Polish group $(X,+)$ has Smital Property if for every dense set $D$ and a Borel $\Ii$-positive set $B$ the algebraic sum $D+B$ is a complement of a set from $\Ii$. We consider several variants of this property and study their connections with countable chain condition, maximality and how well they are preserved via Fubini products.
\end{abstract}

\maketitle

\section{Introduction}
We adopt the usual set-theoretical notation. We say that $X$ is a Polish space if it is a separable and completely metrisable topological space. $\Bor(X)$ denotes the family of Borel subsets of $X$. $\Mm(X)$ and $\Nn(X)$ are families of meager and null subsets of $X$ respectively. Sometimes we will write briefly $\Mm$ and $\Nn$ if the underlying space is clear from the context.
\\
Let $\Aa$ be a ($\sigma$-)algebra and $\Ii$ be a ($\sigma$-)ideal on an Abelian Polish group $(X,+)$. Throughout the paper we assume that $\mc{I}$ contains all singletons and
	\[
		(\forall I\in \Ii) (\exists A\in\Aa\cap \Ii) (I\se A).
	\]
	If $\Aa$ is not explicitly stated we assume $\Aa=\Bor(X)$. In such a case $\Ii$ has a Borel base.
		
We say that a set $A$ is $\Ii$-positive, if $A\notin\Ii$. A is called $\Ii$-residual if $A^c\in\Ii$, we denote this fact by $A\in\Ii^\star$.

For any sets $A, B\se X$ we denote the algebraic sum of these sets by $A+B$, i.e.
\[
	A+B=\{a+b: a\in A,\; b\in B\}.
\]

Let us now recall the classical notion of the Steinhaus property and probably less famous notions of Smital properties which were studied in \cite{BFN}.
\begin{definition}
	We say that a pair $(\Aa,\Ii)$ has
	\begin{enumerate}[(i)]
		\item the Steinhaus Property if for any $A,B\in \Aa\bez \Ii$ the set $A-B$ has a nonempty interior;
		\item the Smital Property, briefly SP, if for every dense set $D$ and every $A\in\Aa\bez\Ii$ the set $A+D$ is $\Ii$-residual;
		\item the Weaker Smital Property, briefly WSP, if there exists a countable and dense set $D$ such that for every $A\in\Aa\bez\Ii$ the set $A+D$ is $\Ii$-residual;
		\item the Very Weak Smital Property, briefly VWSP, if for every $A\in\Aa\bez\Ii$ there is a countable set D such that the set $A+D$ is $\Ii$-residual.
	\end{enumerate}		
\end{definition}

Note that $\Mm$ and $\Nn$ have all of these properties.

The following Proposition seems to be a folklore but we could not find the proof of the second part in the literature.
\begin{proposition}
	The Steinhaus Property is equivalent to SP.
\end{proposition}
\begin{proof}
	The Steinahus Property implies SP. Let $A\in\Aa\bez \Ii$ and $D$ be countable and dense. We may assume that $D$ is a subgroup. Suppose that $(A+D)\notin \Ii^\star$. Then $(A+D)^c\notin\Ii$. By the Steinhaus Property $(A+D)-(A+D)^c$ contains an open neighborhood of $0$. A contradiction since $0\notin(A+D)-(A+D)^c$.
	
	SP implies the Steinhaus Property. Assume that $A-B$ has an empty interior for $A,B\in \Aa\bez\Ii$. Then there is a countable dense set $D\se (A-B)^c$. It follows from SP that $(B+D)\cap A\neq\0$, a contradiction.
\end{proof}

Let $F\subseteq X\times Y$. Then for $x\in X$
	\[
		F_x=\{y\in Y:\; (x,y)\in F\}
	\]
	is a vertical section of $F$ at $x$. Similarly, for $y\in Y$
	\[
		F^y=\{x\in X:\; (x,y)\in F \}
	\]
	is a horizontal section of $F$ at $y$.
	
	For ($\sigma$-)algebras $\Aa\se P(X)$ and $\Bb\se P(Y)$ let $\Aa\otimes\Bb\se P(X\times Y)$ denote the ($\sigma$-)algebra generated by the rectangles of the form $A\times B$ with $A\in\Aa$ and $B\in \Bb$.
\begin{definition}
	Let $(\Aa,\Ii)$ and $(\Bb,\Jj)$ be pairs of ($\sigma$-)algebra - ($\sigma$-)ideal on Polish spaces $X$ and $Y$ respectively. Then we define the Fubini product of $\Ii$ and $\Jj$ as follows:
	\[
		K\in \Ii\otimes\Jj \IFF (\exists C\in \Aa\otimes\Bb)(K\se C \land \{x\in X: C_x\notin \Jj\}\in\Ii).
	\]	
\end{definition}
	Notice that in the case of ideals possessing Borel bases, i.e. $\Aa=\Bor(X)$ and $\Bb=\Bor(Y)$ the above definition ensures the existence of Borel base for $\Ii\otimes\Jj$.
	\begin{proposition}
	If $(\Aa\otimes\Bb,\Ii\otimes\Jj)$ has SP (WSP, VWSP), then $(\Aa,\Ii)$ and $(\Bb,\Jj)$ also have it.
\end{proposition}
\begin{proof}
	Let us consider the case of WSP for $\Ii$. Let $D$ be a witness that $\Ii\otimes\Jj$ has WSP. Let $A\in\Aa\bez\Ii$. The set $R=D+(A\times Y)$ is $\Ii\otimes\Jj$-residual, therefore
	\[
		\widetilde{R}=\{x\in X: R_x \tn{ is $\Jj$-residual}\}
	\]
	is $\Ii$-residual. Clearly, $A+\pi_X(D)=\widetilde{R}$, hence we are done.
\end{proof}
	
	The following definition is a variation of the ones found in \cite[Definition 18.5]{Kech}, \cite{Sri} and agrees with the notation given in \cite{BG}.
\begin{definition}
	Let $X$ and $Y$ be Polish spaces and let $\Ff\se P(X)$, $\Gg\se P(Y)$, $\mc{H}\se P(X\times Y)$ be families of sets. Then we say that $\Gg$ is $\mc{H}$-on-$\Ff$ if for each set $H\in\mc{H}$
	\[
		\{x\in X: H_x\in \Gg\}\in \Ff.
	\]
\end{definition}
Mainly we will be interested in the case where ${\Gg=\Jj\se P(Y)}$ is a $\sigma$-ideal, $\Ff\in\{Bor(X), \sigma(Bor(X)\cup \Ii)\}$ and $\mc{H}\in\{Bor(X\times Y), \sigma(Bor(X\times Y)\cup \Ii\otimes\Jj)\}$. Here $\sigma(\Dd)$ is the $\sigma$-algebra generated by the family $\Dd$. We will write, for example, that $\Jj$ is Borel-on-measurable instead of $Bor(X\times Y)$-on-$\sigma(Bor(X)\cup \Ii)$ if the context is clear. Notice that both $\Mm$ and $\Nn$ are Borel-on-Borel (see \cite[Exercise 22.22, 22.25]{Kech}).
	\begin{example}
		Measurable-on-measurable not necessarily implies Borel-on-Borel.
	\end{example} 
	\begin{proof}
		Let $\Jj=\{\0\}$ and take a Borel set $B$ projection of which is analytic and not Borel.
	\end{proof}
	\begin{proposition}
		Borel-on-measurable implies measurable-on-measurable.
	\end{proposition}
	\begin{proof}
		Assume that $\Jj$ is Borel on measurable. Let $C\se X\times Y$ be measurable with respect to $\Ii\otimes \Jj$. Then $C=(B\bez A_1)\cup A_2$, where $B$ is Borel and $A_1, A_2\in \Ii\otimes\Jj$. Clearly
		\[
			\{x\in X: {A_2}_x\notin\Jj\}\in\Ii,
		\]
		hence it is measurable. See that
		\begin{align*}
			&\{x\in X: (B\bez A_1)_x\notin\Jj\}=\{x\in X: B_x\bez {A_1}_x\notin\Jj\}=
			\\
			&=\{x\in X: B_x\notin\Jj, {A_1}_x\in\Jj\}\cup\{x\in X: B_x\notin\Jj, {A_1}_x\notin\Jj, B_x\bez{A_1}_x\notin\Jj\}.
		\end{align*}
		Since $\{x\in X: B_x\notin\Jj\}$ is measurable, $\{x\in X: {A_1}_x\in\Jj\}\in\Ii^\star$ and
		\[
			\{x\in X: B_x\notin\Jj, {A_1}_x\notin\Jj, B_x\bez{A_1}_x\notin\Jj\}\se \{x\in X: {A_1}_x\notin\Jj\}\in\Ii,
		\]
		the set $\{x\in X: C_x\notin\Jj\}$ is measurable.
	\end{proof}
	
\section{Smital and ccc}

	Let $\Ii$ be a $\sigma$-ideal in a Polish space $X$ and assume that $\Ii$ has a Borel base. We say that $\Ii$ satisfies countable chain condition (briefly: ccc) if every family of pairwise disjoint Borel $\Ii$-positive sets is countable. Let us also recall the following cardinal coefficient 
	\[
		\cov(\Ii)=\min\{|\Aa|: \Aa\se\Ii,\, \bigcup \Aa=X\}.
	\]
\begin{theorem}\label{WSP to prawie CCC}
				Let $\Ii$ be a $\sigma$-ideal possessing WSP. Then $\Ii$ satisfies $\ccc$  or $\cov(\Ii)=\omega_1$.
			\end{theorem}
			\begin{proof}
				Let $\Ii$ be  $\sigma$-ideal with WSP and let $D$ witness it. Let $\{B_\alpha: \alpha<\omega_1\}$ be a family of pairwise disjoint Borel $\Ii$-positive sets. WSP implies that for each $\alpha<\omega_1$ a set $D+B_\alpha$ is $\Ii$-residual. If $\bigcap_{\alpha<\omega_1}(D+B_\alpha)=\0$, then $\cov(\Ii)=\omega_1$. On the other hand, if $\bigcap_{\alpha<\omega_1}(D+B_\alpha)\neq\0$ then for $x\in \bigcap_{\alpha<\omega_1}(D+B_\alpha)$ we have
				\[
					(\forall \alpha\in\omega_1) (\exists d\in D)( x\in d+B_\alpha).
				\] 
				$D$ is countable, hence there exist a set $W\se \omega_1$ of cardinality $\omega_1$ and $d\in D$ such that
				\[
					(\forall \alpha\in W)( x\in d+B_\alpha),
				\]
				which gives $x-d\in \bigcap_{\alpha\in W}B_\alpha$, a contradiction.
			\end{proof}
			The following remark improves the result obtained in \cite{CiSzyWe}.
			\begin{remark}
				Let $\Ii$ be a $\sigma$-ideal possessing WSP. Then the following statements are equivalent:
				\begin{enumerate}[(i)]
					\item For each family of sets $\{B_\alpha: \alpha<\omega_1\}\se\Bb\bez\Ii$ there exists a set $W\se\omega_1$ of cardinality $\omega_1$ such that $\bigcap_{\alpha\in W}B_\alpha\neq\0$;
					\item $\cov(\Ii)>\omega_1$.
				\end{enumerate}
			\end{remark}
			\begin{proof}
				\ti{(ii)} $\IMP$ \ti{(i)} is a part of Theorem \ref{WSP to prawie CCC}. To prove \ti{(i)} $\IMP$ \ti{(ii)} let us suppose that $\cov(\Ii)=\omega_1$. Then there is a family of sets
				\[
					\{A_\alpha: \alpha<\omega_1\}\se\Bb\cap\Ii
				\]
				for which $\bigcup_{\alpha<\omega_1}A_\alpha=X$. Set $\widetilde{A}_\alpha=\bigcup_{\beta\leq\alpha}A_\alpha$ for each $\alpha<\omega_1$. The family $\{\widetilde{A}_\alpha: \alpha<\omega_1\}$ is ascending and covers $X$. Hence $\{\widetilde{A}_{\alpha}^c: \alpha<\omega_1\}$ is descending family of $\Ii$-residual sets. Moreover, for every $W\se\omega_1$ of cardinality $\omega_1$ we have $\bigcap_{\alpha\in W}\widetilde{A}_{\alpha}^c=\0$, which contradicts \ti{(i)}.
			\end{proof}

\section{Preserving Smital properties via products}

In \cite{BFN} the authors present some results on various Smital properites in product spaces. Their setup is, in their words, as general as possible, concerned with algebras and ideals. It is not clear if they intended their results to hold for $\sigma$-algebras and $\sigma$-ideals or algebras and ideals only. The formulation of \cite[Theorem 4.2]{BFN} suggests the former since it is concerned with the Borel algebra and the families of meager and null sets. In their proof they rely implicitly on the following property.
	\begin{definition}
		Let $\Aa\se P(X\times Y)$ be a ($\sigma$-)algebra and let $\Ii\se P(X\times Y)$ be a ($\sigma$-)ideal. A pair $(\Aa, \Ii)$ has the positive rectangle property (PRP) if for every $\Ii$-positive set $A\in \Aa$ there is an $\Ii$-positive rectangle $R$ satisfying $R\se A\cup I$ for some $I\in\Ii$.
	\end{definition}
	In this section we show explicitly that PRP holds for pairs algebra-ideal. However, PRP does not hold for pairs $\sigma$-algebra - $\sigma$-ideal in general, including the relevant here pair of Borel $\sigma$-algebra and the family of null sets.

\begin{example}
	The pair $(\Bor(\RR^2), [\RR]^{\leq\omega}\otimes[\RR]^{\leq\omega})$ does not have PRP.  
\end{example}
\begin{proof}
	Let $P\se \RR$ be a perfect set such that $P\cap (P+x)$ is at most 1-point for $x\neq 0$ (see \cite{MZ}). Let us set
	\[
		B=\{(x,y): x\in P \land y\in P-x\}\bez(\RR\times\{0\}).
	\]
	$B$ is Borel and $B\notin[\RR]^{\leq\omega}\otimes[\RR]^{\leq\omega}$. If $x\in B^y$, then $x\in P$ and $x\in P-y$, therefore $B^y$ is at most 1-point.
	\\
	Let us suppose that there are sets $A_1, A_2\in\Bor(\RR^2)\bez [\RR]^{\leq\omega}$ and a set $K\in [\RR]^{\leq\omega}\otimes[\RR]^{\leq\omega}$ such that $(A_1\times A_2)\bez K\se B$. Let $T=\{x\in\RR: |K_x|>\omega\}$ and notice that $|T|\leq \omega$. Pick $x_0\in A_1\bez T$. Then $K_{x_{0}}$ is countable. By the definition of $B$ for each $y\in A_2\bez K_{x_0}$ we have $K^y\es A_1\bez (\{x_0\}\cup T)$. So $A_2\bez K_{x_0}\se K_{x_1}$ for $x_1\in A_1\bez (\{x_0\}\cup T)$, a contradiction.
	\end{proof}
It is clear that the pair $(\Bor(\RR^2), \Mm)$ has PRP, since every nonmeager set possessing the property of Baire is nonempty and open, modulo a set of the first category. What about $(\Bor(\RR^2), \Nn)$? As a warm up let us recall the following folklore result.
\begin{proposition}
	Every set $E\se [0,1]^2$ of positive measure contains a subset of the same measure which does not contain a rectangle of positive measure.
\end{proposition}
\begin{proof}
	Let $E\se [0,1]^2$ have positive measure. Consider $E'=\{(x,y)\in E: x-y\in\QQ^c\}$. To see that $\lambda(E')=\lambda(E)$ let us observe that $E'_x=E_x\cap (\QQ^c+x)$ for every $x\in[0,1]$ and $\QQ^c$ is co-null. Now, if $A\times B\se E'$, $A$ and $B$ of positive measure, then $A-B$ should contain a nonempty open set (Steinhaus Theorem), but clearly $\QQ\cap (A-B)=\0$. A contradiction completes the proof.
\end{proof}
This result may be improved with the following Lemma.
\begin{lemma}\label{zbior i jego dopelnienie tnie pozytywnie kazdy open}
	There exists a set $F\se \RR$ such that $\lambda(F\cap U)>0$ and $\lambda(F^c\cap U)>0$ for every nonempty open set $U$.
\end{lemma}
\begin{proof}
	Let $(B_n: n\in\omega)$ be an enumeration of the basis of $\RR$. At the step $0$ let $C^1_0, C^2_0 \se B_0$ be two disjoint Cantor sets of positive measure. At the step $n+1$ let assume that we have two sequences of pairwise disjoint Cantor sets $(C^1_k: k\leq n)$ and $(C^2_k: k\leq n)$ which for all $i\in\{0,1\}$ and $k\leq n$ satisfy $\lambda(C^i_k\cap B_k)>0$. The set
	\[
		B_{n+1}\bez \bigcup_{k\leq n}(C^1_k\cup C^2_k)
	\]
	is nonempty and open, hence it contains two disjoint Cantor sets of positive measure. Denote them by $C^1_{n+1}$ and $C^2_{n+1}$. This completes the construction and $F=\bigcup_{n\in\omega}C^1_k$ is the desired set.
\end{proof}

The above Lemma will serve as a tool to prove the result from \cite{EO}.
\begin{example}[Erd\"os, Oxtoby]
	There is a set $E\se \RR^2$ such that $E\cap (A\times B)$ and $E^c\cap (A\times B)$ have positive measure for each $A, B\se \RR$ of positive measure.
\end{example}
\begin{proof}
	Let $F$ be as in the formulation of Lemma \ref{zbior i jego dopelnienie tnie pozytywnie kazdy open} and set ${E=\{(x,y)\in\RR^2: x-y\in F\}}$. Let $A, B\se\RR$ have a positive measure. Then
	\begin{align*}
		\lambda(E\cap A\times B)=\iint\chi_{F}(x-y)\chi_{A}(x)\chi_{B}(y)dxdy= \iint\chi_{F}(x)\chi_A(x+y)\chi_{B}(y)dxdy=
		\\
		=\int_F \lambda((A-x)\cap B)dx.
	\end{align*}
	$\lambda((A-x)\cap B)$ is a continuous non-negative function. Furthermore, it is positive on some interval, since $\int_\RR \lambda((A-x)\cap B)dx=\lambda(A\times B)$, so $\int_F \lambda((A-x)\cap B)dx>0$.
\end{proof}

\begin{corollary}
	$(\Bor({\RR^2}), \Nn)$ does not have PRP.
\end{corollary}
	Now we will focus on PRP for products of algebras.
\begin{lemma}
	Let $\Aa\se P(X)$ and $\Bb\se P(Y)$ be algebras. Then $\Aa\otimes\Bb\se P(X\times Y)$ consists of finite unions of rectangles.
\end{lemma}
\begin{proof}
	First let us observe that complements of rectangles are finite unions of rectangles:
	\[
		(A\times B)^c=(A\times B^c)\cup (A^c\times B)\cup (A^c\times B^c).
	\]
	Next, see that finite intersection of a finite union of rectangles is again a finite union of rectangles. Let $T_1$, $T_2\se \omega$ be finite. Let $A^n_k$, $B^n_k$ be rectangles from $\Aa$ and $\Bb$ respectively for every $n\in T_1$ and $k\in T_2$. Then
	\[
		\bigcap_{n\in T_1}\bigcup_{k\in T_2}A^n_k\times B^n_k=\bigcup_{f\in T_2^{T_1}}\bigcap_{n\in T_1}A^{n}_{f(n)}\times B^n_{f(n)}.
	\]
	$T_2^{T_1}$ is finite and finite intersections of rectangles are also rectangles, hence the proof is complete.
\end{proof}
	For algebra $\Aa$ and ideal $\Ii$ let $\Aa[\Ii]$ denote the algebra generated by $\Aa\cup\Ii$.
\begin{proposition}
	Let $\Aa\se P(X)$ and $\Bb\se P(Y)$ be algebras and let $\Ii$ be an ideal in $X\times Y$. Then $((\Aa\otimes\Bb)[\Ii],\Ii)$ has PRP. 
\end{proposition}
\begin{proof}
	Notice that $(\Aa\otimes\Bb)[\Ii]=\{C\symdif I: C\in\Aa\otimes \Bb,\; I\in\Ii\}$. PRP follows from the previous Lemma.
\end{proof}
	
	From now on let $\Aa\se P(X), \Bb\se P(Y)$ be $\sigma$-algebras, and $\Ii\se P(X)$, $\Jj\se P(Y)$ $\sigma$-ideals.	
	
	\begin{theorem}
	Let $\Ii$ and $\Jj$ possess WSP and assume one of the following properties
	\begin{enumerate}[(i)]
		\item $\Jj$ is Borel-on-Borel;
		\item $\Jj$ measurable-on-measurable;
		\item $(\Bor(X\times Y), \Ii\otimes\Jj)$ has PRP.
	\end{enumerate}		
	Then $\Ii\otimes \Jj$ also has WSP.
	\end{theorem}
	\begin{proof}
		Let $D_1$ and $D_2$ witness WSP for $\Ii$ and $\Jj$ respectively. Let $B\in\Bor(X\times Y)\bez \Ii\otimes \Jj$. If any of the properties \ti{(i)-(iii)} holds then a set $\widetilde{B}=\{x\in X: B_x\notin\Jj\}$ contains a Borel, $\Ii$-positive set and $D_1+\widetilde{B}$ is $\Ii$-residual. Let us observe that
		\[
			(D_1\times D_2)+B\es\bigcup_{d_1\in D_1}\bigcup_{x\in\widetilde{B}}(\{d_1+x\}\times (D_2+B_x)),
		\]
		therefore for every $x\in D_1+\widetilde{B}$ the set $((D_1\times D_2)+B)_x$ is $\Jj$-residual. Since $D_1+\widetilde{B}$ is $\Ii$-residual, the proof is complete.
	\end{proof}
	
	In \cite{BalKot} the authors showed that $\Mm\otimes \Nn$ and $\Nn\otimes \Mm$ have SP and thus WSP. The following corollary extends this result regarding WSP.
	\begin{corollary}
		Let $n\in\omega$ and $\Ii_k\in \{\Mm, \Nn\}$ for any $k\leq n$. Then $\Ii_0\otimes \Ii_1 \otimes ... \otimes \Ii_n$ has WSP.
	\end{corollary}
	\begin{proof}
		By \cite[Lemma 3.1]{CieMi} the ideal $\Ii_0\otimes \Ii_1 \otimes ... \otimes \Ii_n$ is Borel-on-Borel for any $n\in\omega$ and $\Ii_k\in \{\Mm, \Nn\}$, $k\leq n$.
	\end{proof}
	In \cite[Theorem 4.3]{BFN} the authors showed that if $\Bb=\Jj\cup\Jj^\star$ and $(\Aa, \Ii)$ has SP then $(A\otimes \Bb, \Ii\otimes\Jj)$ also has SP. We will generalize this result (Theorem \ref{generalizacja}). Let us start with two technical definitions.
	
\begin{definition}
	We say that a pair $(\Aa\otimes\Bb, \Ii\otimes\Jj)$ has the Tall Rectangle Hull Property (TRHP) if	for every set $C\in\Aa\otimes\Bb$
	\[
		(\exists \widetilde{C}\in \Aa, I \in \Ii, J \in \Jj)((\widetilde{C}\bez I)\times(Y\bez J)\se C \se (\widetilde{C}\times Y)\cup (I \times Y)\cup (X \times J)).
 \]
 If a set $C$ fulfills the above condition we will say that it has TRHP witnessed by the triple $(\widetilde{C}, I, J)$.
 
 Analogously we define Wide Rectangle Hull Property (WRHP):
 	 \[
		(\exists \widetilde{C}\in \Bb, I \in \Ii, J \in \Jj)((X\bez I)\times(\widetilde{C}\bez J)\se C \se (X\times\widetilde{C})\cup (I \times Y)\cup (X \times J)).
	 \] 
\end{definition}

\begin{proposition}
	If a pair $(\Aa\otimes\Bb, \Ii\otimes\Jj)$ have TRHP or WRHP then it has PRP.
\end{proposition}
\begin{proof}
	Let $C\in\Aa\otimes\Bb$ has TRHP witnessed by $(\wt{C}, I, J)$:
	\[
		(\wt{C}\bez I)\times (Y\bez J)\se C \se (\widetilde{C}\times Y)\cup (I \times Y)\cup (X \times J))
	\]
	and assume that $C\notin\Ii\otimes\Jj$. Then $\wt{C}\times Y$ is the desired rectangle. Clearly $\wt{C}\times Y\se C$ modulo a set from $\Ii\otimes\Jj$. It is also $\Ii\otimes\Jj$-positive, otherwise
	\[
		(\widetilde{C}\times Y)\cup (I \times Y)\cup (X \times J))\in\Ii\otimes\Jj
	\]
	and also $C\in\Ii\otimes\Jj$.
	
	The proof of WRHP case is almost identical.
\end{proof}

\begin{lemma}\label{TRHP closed under countable unoins and complements}
	The family of sets possessing TRHP is closed under countable unions and complements. The same is true for the family of sets possessing WRHP. 
\end{lemma}
\begin{proof}
	Proofs for both cases follow the same pattern, so without loss of generality let us focus on the case of TRHP.
	
	Let $C=\bigcup_{n\in\omega}C_n$ and $(\widetilde{C}_n, I_n, J_n)$ witness TRHP for $C_n$, $n\in\omega$. Then for each $n \in\omega$
	\[
		(\wt{C}_n\bez I_n)\times(Y\bez J_n)\se C_n \se (\wt{C}_n\times Y)\cup (I_n\times Y)\cup ( X\times J_n).
	\]
	Then
	\begin{align*}
		(\bigcup_{n\in\omega}\wt{C}_n\bez\bigcup_{n\in\omega}I_n)\times (Y\bez \bigcup_{n\in\omega}J_n)&\se\bigcup_{n\in\omega}(\wt{C}_n\bez I_n)\times(Y\bez J_n)\se \bigcup_{n\in\omega}C_n=C
		\\
		&\se \bigcup_{n\in\omega}\big((\wt{C}_n\times Y)\cup (I_n\times Y)\cup ( X\times J_n)\big)
		\\
		&\se ((\bigcup_{n\in\omega}C_n)\times Y) \cup ((\bigcup_{n\in\omega}I_n)\times Y)\cup (X\times(\bigcup_{n\in\omega}J_n)).
	\end{align*}
	Hence, setting $\wt{C}=\bigcup_{n\in\omega}\wt{C}_n$, $I=\bigcup_{n\in\omega}I_n$, $J=\bigcup_{n\in\omega}J_n$ completes this part of the proof.
	
	Now let $C=D^c$ for $D$ witnessing TRHP with $(\wt{D}, I, J)$. We have
	\[
		(\wt{D}\bez I) \times (Y\bez J)\se D \se (\wt{D}\times Y) \cup (I\times Y) \cup (X\times J).
	\]
	Through complementation
	\[
		\big((\wt{D}\times Y) \cup (I\times Y) \cup (X\times J)\big)^c\se C \se \big((\wt{D}\bez I) \times (Y\bez J)\big)^c.
	\]
	Let us focus on the right-hand side
	\[
		\big((\wt{D}\bez I) \times (Y\bez J)\big)^c=((\wt{D}\bez I)\times (Y\bez J)^c) \cup ((\wt{D}\bez I)^c\times Y)\se ((\wt{D})^c\times Y)\cup(X\times J)\cup (I\times Y).
	\]
	Now the left-hand side. It is an intersection of the following sets
	\begin{align*}
		(\wt{D}\times Y)^c&=(\wt{D})^c\times Y,
		\\
		(I\times Y)^c&=(X\bez I) \times Y,
		\\
		(X\times J)^c&=X\times (Y\bez J),
	\end{align*}
	which is equal to
	\[
		((\wt{D})^c\bez I) \times (Y\bez J).
	\]
	In summary
	\[
		((\wt{D})^c\bez I) \times (Y\bez J)\se C \se ((\wt{D})^c\times Y)\cup(X\times J)\cup (I\times Y).
	\]
	Then $((\wt{D})^c, I, J)$ witnesses TRHP for $C$. The proof is complete.
\end{proof}

\begin{theorem}
	Let $\Cc=\Aa\otimes\Bb$. Then
	\begin{enumerate}[(i)]
		\item if $\Aa=\Ii\cup\Ii^\star$ then $(\Cc, \Ii\otimes\Jj)$ has WRHP.
		\item if $\Bb=\Jj\cup\Jj^\star$ then $(\Cc, \Ii\otimes\Jj)$ has TRHP.
	\end{enumerate}		
\end{theorem}
\begin{proof}
	Let us prove $(ii)$ (the proof of $(i)$ is similar). Let us notice that rectangles from $\Cc$ have TRHP. Indeed, if $C=A\times B$ then set
	\begin{align*}
		\wt{C}&=A,\; I=\0,\; J=B^c \tn{ if } B\in\Jj^\star
		\\
		\wt{C}&=\0,\; I=\0,\; J=B \tn{ if } B\in\Jj.
	\end{align*}
	The rest of the proof relies on Lemma \ref{TRHP closed under countable unoins and complements} which allows us to perform an induction over the hierarchy of sets making up the $\sigma$-algebra $\Cc$.
\end{proof}


\begin{theorem}\label{generalizacja}
	Let $\Cc=\Aa\otimes \Bb$ and assume that 
	\begin{enumerate}[(i)]
		\item $(\Cc, \Ii\otimes\Jj)$ has TRHP and $(\Aa, \Ii)$ has SP, or
		\item $(\Cc, \Ii\otimes\Jj)$ has WRHP and $(\Bb, \Jj)$ has SP.
	\end{enumerate}
	Then $(\Cc,\Ii\otimes\Jj)$ has SP.
\end{theorem}
\begin{proof}
	Assume $(i)$. Let $D\se X\times Y$ be dense, set $D_1=\pi_1(D)$ and let $B\in \Cc$ be $\Ii\otimes\Jj$-positive. Then there are $\widetilde{B}\in\Aa\bez \Ii$ and $J\in\Jj$ such that $\widetilde{B}\times (Y\bez J)\se B$. It follows that
	\[
		D+B\es D+(\widetilde{B}\times (Y\bez J))\es \bigcup_{d_1\in D_1}\bigcup_{d_2\in D_{d_1}}(d_1+\widetilde{B})\times (d_2+Y\bez J)).
	\]
	Therefore for every $x\in D_1+\widetilde{B}$ the set $(D+B)_x$ contains a translation of $Y\bez J$. By SP $D_1+\widetilde{B}\in \Ii^{\star}$ thus $D+B$ is $\Ii\otimes\Jj$-residual.
	
	The proof assuming $(ii)$ is analogous.
\end{proof}

\section{Maximal invariant $\sigma$-ideals with Borel bases}

There is a surprising connection between maximal invariant $\sigma$-ideals with Borel bases and Smital properties.

\begin{proposition}\label{VWSP-maximality}
	The following are equivalent:
	\begin{enumerate}[(i)]
		\item $\Ii$ has VWSP; 
		\item $\Ii$ is maximal among invariant proper $\sigma$-ideals with Borel base.
\end{enumerate} 
\end{proposition}
\begin{proof}
	$(i)\IMP (ii):$ Let us suppose that $\Jj\nes\Ii$ is such an ideal. Let $A$ be a Borel set from $\Jj\bez\Ii$. Then there exists a countable set $D$ such that $D+A$ is $\Ii$-residual, therefore $\Jj$-residual, hence $\Jj$ is not proper.
	
	$(ii)\IMP (i):$ Let us suppose that there is a set $B\in Bor\bez \Ii$ for which $B+D$ is not $\Ii$-residual for every countable set $D$. Consider a family
	\[
		\Ii'=\{A\cup C: A\in\Ii \land (\exists D)(C\se D+B \land |D|\leq\omega)\}.
	\]
	It is an invariant proper $\sigma$-ideal satisfying $\Ii\sen\Ii'$, which brings a contradiction.
\end{proof}

\begin{proposition}
	Let $\{\Ii_n: n\in\omega\}$ be a countable family of pairwise distinct maximal invariant $\sigma$-ideals on $X$ with Borel bases. Then for each $n\in\omega$ the $\sigma$-ideal $\Ii_n$ is orthogonal to $\bigcap_{k\in\omega\bez\{n\}}\Ii_k$.
\end{proposition}
\begin{proof}
	Fix $n\in\omega$.	There are sets $A_k\in (\Bor(X)\cap\Ii_n)\bez\Ii_k$, $k\in\omega\bez\{n\}$. By Proposition \ref{VWSP-maximality} for each of them there is countable set $C_k$ such that $A_k+C_k\in \Ii_n\cap \Ii_k^\star$. Therefore
	\[
		\bigcup_{k\in\omega\bez\{n\}}(A_k+C_k)\in \Ii_n\cap \bigcap_{k\in\omega\bez\{n\}}\Ii_k^\star.
	\]
\end{proof}

\begin{corollary}
	Let $\Ii$ be a maximal invariant $\sigma$-ideal with a Borel base different from $\Mm$ and $\Nn$. Then there is $A\in\Ii\cap(\Mm\cap\Nn)^\star$.
\end{corollary}
	Let us now focus on $X=2^\omega$. The following result incorporates techniques similar to these used in \cite[Theorem 3.1]{Zak}.
\begin{theorem}
	There are $\c$ many maximal invariant $\sigma$-ideals on $2^\omega$.
\end{theorem}
\begin{proof}
	Let $\{A_\alpha: \alpha<\c\}$ be an AD family on $\omega$, i.e. for all distinct  $\alpha, \beta<\c$ the set $A_\alpha\cap A_\beta$ is finite. For every $\alpha<\c$ set
	\begin{align*}
		\Ii_\alpha=\{A\se 2^\omega:\; &(\exists B\in \Bor(2^\omega)A\se B\land(\exists M\in\Mm(2^{A_\alpha}))(\forall x\in 2^{A_\alpha})
		\\
		&(x\notin M \to B^{\alpha}_x\in\Nn(2^{\omega\bez A_\alpha}))\},
	\end{align*} 
	where $B^\alpha_x=\{y\rest \omega\bez A_\alpha:\; y\in B,\; y\rest A_\alpha=x\}$.
	
	Let us show that $I_\alpha\neq I_\beta$ if $\alpha\neq \beta$. Set
	\begin{align*}
		&M\in \Mm(2^{A_\alpha\bez A_\beta})\bez \Nn(2^{A_\alpha\bez A_\beta}),
		\\
		&C=\{x\in 2^\omega: x\rest (A_\alpha\bez A_\beta)\in M\}.
	\end{align*}
	Notice that $C\in I_\alpha$. The set
	\[
		M_\alpha=\bigcup_{t\in 2^{A_\alpha\cap A_\beta}}\{y\in 2^{A_\alpha}: y\rest (A_{\alpha}\cap A_{\beta})=t\land y\rest(A_\alpha \bez A_\beta)\in M\}
	\]
	is a finite union of meager sets, hence meager. For each $x\in 2^{A_\alpha}\bez M_\alpha$ the set $C^\alpha_x$ is empty.
	\\
	On the other hand for each $x\in 2^{A_\beta}$
	\[
		C^\beta_x=\{y\in 2^{\omega\bez A_\beta}: y\rest (A_\alpha\bez A_\beta)\in M\}.
	\]
	The above set may be considered as a product $M\times 2^{\omega\bez (A_\alpha \cup A_\beta)}$ of non-null set and the whole space, which is not null.
	
	Now we will show that every $\Ii_\alpha$ is maximal among invariant $\sigma$-ideals on $2^\omega$ with Borel bases.
	Each $\Ii_\alpha$ is essentially $\Mm(2^{A_\alpha})\otimes\Nn(2^{\omega\bez A_\alpha})$. It follows that $\Ii_\alpha$ has WSP, thus by Proposition \ref{VWSP-maximality} the proof is complete.
	\end{proof}
	The reasoning in the above Theorem does not translate for the case of $\RR$. We may ask the following question.
\begin{question}\label{question M N only maximal}
	Are $\Mm$ and $\Nn$ the only maximal invariant $\sigma$-ideals with Borel bases in $\RR$?
\end{question}

\begin{question}
	Is it true that for every set $G\in(\Mm\cap\Nn)^\star$ there is a countable set $C$ such that $C+G=\RR$?
\end{question}
Positive answer to this question would also answer positively Question \ref{question M N only maximal}.

\end{document}